\pdfoutput=1
\documentclass[a4paper,english]{amsart}

\usepackage[sc,osf]{mathpazo}
\DeclareTextFontCommand{\textsl}{\fontfamily{ppl}\fontshape{sl}\selectfont}
\usepackage[scaled=.95]{helvet} 
\usepackage[euler-digits]{eulervm}
\usepackage{microtype}

\usepackage{mathtools}
\usepackage{amssymb}
\usepackage{amsthm}
\usepackage[foot]{amsaddr}
\usepackage{ifthen}
\usepackage{nicefrac}
\usepackage{color}
\usepackage[
pdftitle={One Lie group to define them all},
pdfauthor={Annalisa Conversano, Marcello Mamino},
pdfkeywords={Keywords},
colorlinks=false]{hyperref}

\title{One Lie group to define them all}
\author[A.~Conversano] {\lsstyle Annalisa Conversano}  
\author[M.~Mamino]{\lsstyle Marcello Mamino}

 \address{Annalisa Conversano, Massey University Auckland, New Zealand} 
 \address{Marcello Mamino, Universit\`a di Pisa, Italy}

\date{\mydate{14}{vii}{2021}}

\makeatletter
\def\Hy@Warning#1{}
\makeatother

\makeindex

\makeatletter
\def\@setemails{% 
\ifnum\theg@author > 1 
\mbox{{\itshape E-mail addresses}:\space}{\ttfamily\emails}. 
\else 
\mbox{{\itshape E-mail address}:\space}{\ttfamily\emails}. 
\fi% 
}
\makeatother

\makeatletter
\def\ps@plain{\ps@empty
  \def\@oddfoot{\normalfont\normalsize \hfil\thepage\hfil}%
  \let\@evenfoot\@oddfoot}
\def\ps@firstpage{\ps@plain
  \def\@oddfoot{\normalfont\normalsize \hfil\thepage\hfil
     \global\topskip\normaltopskip}%
  \let\@evenfoot\@oddfoot
  \def\@oddhead{\@serieslogo\hss}%
  \let\@evenhead\@oddhead % in case an article starts on a left-hand page
}
\def\ps@headings{\ps@empty
  \def\@evenhead{%
    \setTrue{runhead}%
    \normalfont\normalsize
    \rlap{\thepage}\hfil
    \textsc{\lsstyle\MakeLowercase{\shortauthors}\hfil}}%
  \def\@oddhead{%
    \setTrue{runhead}%
    \normalfont\normalsize \hfil
    \textsc{\lsstyle\MakeLowercase{\rightmark{}{}}}\hfil\llap{\thepage}}%
  \let\@mkboth\markboth
}
\makeatother
\makeatletter
\let\oldupchars\upchars@
\def\upchars@{\oldupchars\def\-{\U-}}
\makeatother

\def\mathshift{$}
\catcode`$=13
\def${\ifinner\expandafter\mathshift\else\expandafter\myshift\fi}
\def\myshift#1${\raisebox{0ex}[0ex][0ex]{\mathshift#1\mathshift}}
\AtBeginDocument{\catcode`$=13}

\makeatletter
\let\oldsection\section
\def\newsection#1{\oldsection{\lsstyle #1}}
\def\newsectionr#1{\oldsection*{\lsstyle #1}}
\def\section{\@ifstar\newsectionr\newsection}
\makeatother

\makeatletter
\newtheorem{ghost@theorem}{}[section]
\def\@maketheorem#1=#2;{
	\newtheorem{#1}[ghost@theorem]{#2}}
\def\maketheorem#1{
	\@for\@x:=#1\do{
		\expandafter\@maketheorem\@x;}}
\makeatother

\def\eqdef{\stackrel{\textsl{\tiny def}}{=}}

\def\Z{\mathbb Z}
\def\R{\mathbb R}

\def\SO{\operatorname{SO}}

\def\st{\;:\;}

\def\quotient#1#2{\mathchoice
	{#1/\raisebox{-.5ex}{$\mathsurround=0pt\displaystyle #2$}}
	{#1/\raisebox{-.5ex}{$\mathsurround=0pt\textstyle #2$}}
	{#1/\raisebox{-.3ex}{$\mathsurround=0pt\scriptstyle #2$}}
	{#1/\raisebox{-.1ex}{$\mathsurround=0pt\scriptscriptstyle #2$}}}

\def\MMfirstoftwo#1#2{#1}
\def\MMsecondoftwo#1#2{#2}

\def\MMendlist{\MMendlist}
\def\MMuniquetag{\MMuniquetag}

\def\checkempty#1{\MMdocheckempty#1\MMuniquetag\MMendlist}
\def\MMdocheckempty#1#2\MMendlist{\ifx#1\MMuniquetag\expandafter\MMfirstoftwo\else\expandafter\MMsecondoftwo\fi}

\let\n\oldstylenums

\def\mydate#1#2#3{\hbox{\n{#1}$\cdot${\scshape #2}$\cdot$\n{#3}}}
\def\new#1#{\MMdonew{#1}}
\def\MMdonew#1#2{\checkempty{#1}{\MMnewtwoparms[#2]{#2}}{\MMnewtwoparms#1{#2}}}
\def\MMnewtwoparms[#1]#2{\textsc{#2}\index{#1}}
\def\-{\nobreakdash-\hspace{0pt}}
\def\U-{\raise0.2ex\hbox{-}}
\def\url#1{\href{#1}{url\nobreakdash---\texttt{#1}}}

\def\arXiv#1{\href{http://arxiv.org/abs/#1}{\texttt{arXiv:#1}}}

\makeatletter
\def\eatspace#1{#1}
\def\myitem#1{\hfil\break\hbox to 0em{\hss#1. }\def\@currentlabel{#1}\eatspace}
\makeatother

\theoremstyle{plain}
\newtheorem*{theorem*}{Theorem}
\maketheorem{%
	theorem=Theorem,%
	lemma=Lemma,%
	proposition=Proposition,%
	corollary=Corollary,%
	fact=Fact%
}
\newtheorem{claim-toplemma}{Claim}
\theoremstyle{definition}
\maketheorem{%
	definition=Definition,%
	assumption=Assumption,%
	example=Example%
}
\theoremstyle{remark}
\maketheorem{%
	observation=Observation,%
	remark=Remark,%
	claim=Claim%
}

\def\SO{SO}

\newcommand{\book}[2]{{\scshape#1}, {\bf #2}}
\newcommand{\pre}[2]{{\scshape#1}, #2}
\newcommand{\publ}[6]
{{\scshape#1}, #2, {\itshape #3}, {\bf #4} (#5), pp.~#6.}

\begin{document}

\begin{abstract}
We produce a connected real Lie group that, as a first order
structure in the group language, interprets the real field expanded with a
predicate for the integers. Moreover, the domain of our interpretation is
definable in the group.
\end{abstract}

\maketitle
\thispagestyle{empty}
 
\section{Introduction}

\noindent The purpose of this note is to exhibit a connected Lie group
which is, from the point of view of model theory, as badly behaved as it
possibly can.
In particular, we produce a Lie group -- which happens to be connected, nilpotent and
non-compact -- that, seen as a first order
structure in the group language,
interprets the real field expanded with a predicate for the
integers~$(\R,+,\times,\Z)$.
This, in turn, can be considered a \textit{wild} structure from the point
of view of model theory, and, in particular, every Lie group is
interpretable in it.

Indeed, we claim a slightly stronger statement, namely that there is a
connected Lie group that {\it defines} the real field with a predicate for
the integers, and thus every Lie group. To clarify this second claim, we
must specify what is intended for a structure~$\mathcal{A}$ to define a
structure~$\mathcal{B}$. A structure is a collection of
sets: the domain, the relations, the functions in its signature.
Hence, it makes sense to say that~$\mathcal{B}$ is definable
in~$\mathcal{A}$ whenever the domain~$B$ of~$\mathcal{B}$ is a definable set in~$\mathcal{A}$,
and all relations and functions in the signature
of~$\mathcal{B}$ are definable sets in~$\mathcal{A}$. More generally, we
say that $\mathcal{B}$ is definable in~$\mathcal{A}$ when there is a
structure $\mathcal{B}'$ isomorphic to~$\mathcal{B}$ which is definable
in~$\mathcal{A}$ in the sense above.
In other words, $\mathcal{A}$ defines~$\mathcal{B}$ precisely when
$\mathcal{A}$ interprets~$\mathcal{B}$ in such a way that the domain of
the interpretation is a definable set in~$\mathcal{A}$.
Thus, for instance, we can say that
the group $\SO_2(\R)$ is definable in~$(\R,+,\cdot)$, meaning
that there is a set~$G$ definable in~$(\R,+,\cdot)$ and a
function~$\_\cdot\_\colon G\times G \to G$ such that~$(G,\cdot)$ is
isomorphic, as a first order structure (or, equivalently, as a group)
to~$\SO_2(\R)$.

We mention here that {\it definable} or {\it interpretable}, in this
work, always means definable or interpretable with parameters.

%Recall that a structure~$A$ is interpretable in a structure~$B$
%when the underlying set of~$A$ is in bijection with a set definable in~$B$
%modulo an equivalence relation, itself definable in~$B$, in such a way that the
%relations of~$A$ map to relations definable in~$B$. 
%This notion is weaker than $A$ being definable in~$B$, in that, for definability, $A$ must be in bijection with a set definable in~$B$ tout court, without equivalence
%relation.  

The immediate motivation for this note has been a question by
Antongiulio Fornasiero, about the possibility to interpret every connected Lie
group in a d-minimal structure (for results in this context
see~\cite{fo}), which is ruled out by our example.
More in general, there is a growing number of
classes of Lie groups known to enjoy nice model-theoretical properties, such as:
Nash groups~\cite{pillay88},
algebraic groups~\cite{pps2},
compact~\cite{nesin-pillay, nipI} or semisimple~\cite{pps1} Lie groups,
and covers of all the above~\cite{bm, hpp2}.  
At the same time, model theoretic methods are being
used to attack classical problems~\cite{bpp}.
Our result provides a negative example, suggesting that there are
non-compact
Lie groups that, despite being connected and nilpotent, are intractable for model theory. This contrasts with the
compact case, in fact every compact Lie group is isomorphic to a group definable
in the real field, and, in turn, the real field can be recovered from any
semisimple Lie group~\cite{pps1}.
To study the class of Lie groups up to first order definability in more
detail is a complicated and possibly interesting task; it is, however, beyond the scope of this
note: all we intend to say is that this class has a maximal element.
\vskip\baselineskip

\section{Construction of the group}

\noindent Recall that a real Lie group -- for short {\it Lie group} -- is a real analytic manifold equipped with a real
analytic group operation. Real manifolds are assumed to be second countable and
Hausdorff.

\begin{fact}\label{coll}
Identify the field~$\R$ with the subset~$\{0\}\times\R$ of~$\R^2$.
Then the field operations of~$\R$
can be defined using only the incidence graph of the straight lines
in~$\R^2$. More precisely, the field operations are first order definable
in the structure~$(\R^2,\textrm{coll})$ where $\textrm{coll}(p,q,r)$ denotes the
ternary collinearity relation.
\end{fact}
\begin{proof}
First observe that, working in~$(\R^2,\textrm{coll})$, we can identify
a straight line by a pair of different points, and we can tell whether
the straight lines identified by two pairs of points coincide, meet, or are
parallel, by means of first order formulas. We can thus implement constructions based on the notions of
incidence and parallelism.

It is an ancient observation that the {\it arithmetic of segments} can be
defined geometrically: geometrical constructions to this effect can be found, for
instance, in the
works of Descartes~\cite{descartes} and Hilbert~\cite[\S~15]{grundlagen}.
We need constructions, however, that rely solely on
incidence and parallelism. These are classically called {\it von Staudt
constructions}, with reference to~\cite[\S~19]{staudt}, despite the fact
that, formally, this work is
set in the context of projective geometry. The affine versions of the von Staudt
constructions are depicted in the figure below: the reader will work out the
details without any difficulty.\\
\def\prd{$x\cdot y$}\hbox to \textwidth{\hss%% Creator: Inkscape inkscape 0.92.4, www.inkscape.org
%% Accompanies image file 'figure.pdf' (pdf, eps, ps)
\begingroup%
  \makeatletter%
  \providecommand\color[2][]{%
    \errmessage{(Inkscape) Color is used for the text in Inkscape, but the package 'color.sty' is not loaded}%
    \renewcommand\color[2][]{}%
  }%
  \providecommand\transparent[1]{%
    \errmessage{(Inkscape) Transparency is used (non-zero) for the text in Inkscape, but the package 'transparent.sty' is not loaded}%
    \renewcommand\transparent[1]{}%
  }%
  \providecommand\rotatebox[2]{#2}%
  \newcommand*\fsize{\dimexpr\f@size pt\relax}%
  \newcommand*\lineheight[1]{\fontsize{\fsize}{#1\fsize}\selectfont}%
  \ifx\svgwidth\undefined%
    \setlength{\unitlength}{348.44812295bp}%
    \ifx\svgscale\undefined%
      \relax%
    \else%
      \setlength{\unitlength}{\unitlength * \real{\svgscale}}%
    \fi%
  \else%
    \setlength{\unitlength}{\svgwidth}%
  \fi%
  \global\let\svgwidth\undefined%
  \global\let\svgscale\undefined%
  \makeatother%
  \begin{picture}(1,0.23613647)%
    \lineheight{1}%
    \setlength\tabcolsep{0pt}%
    \put(0,0){\includegraphics[width=\unitlength,page=1]{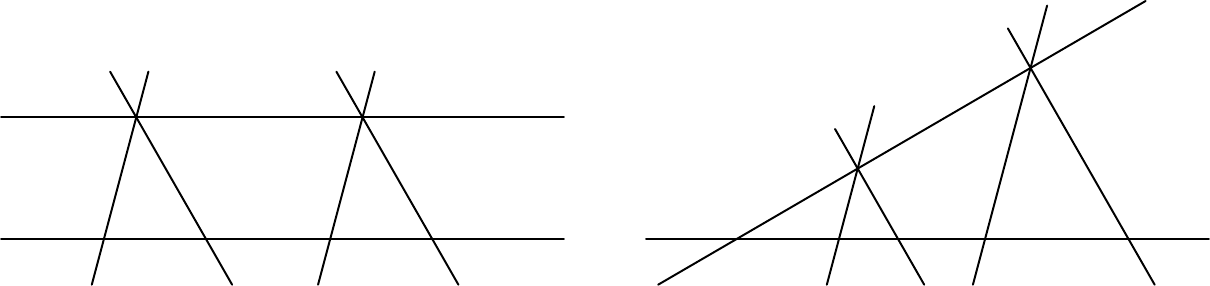}}%
    \put(0.08659,0.01387771){\color[rgb]{0,0,0}\makebox(0,0)[lt]{\lineheight{10}\smash{\begin{tabular}[t]{l}$0$\end{tabular}}}}%
    \put(0.18728573,0.01387771){\color[rgb]{0,0,0}\makebox(0,0)[lt]{\lineheight{10.00000095}\smash{\begin{tabular}[t]{l}$x$\end{tabular}}}}%
    \put(0.27497351,0.01387771){\color[rgb]{0,0,0}\makebox(0,0)[lt]{\lineheight{10.00000095}\smash{\begin{tabular}[t]{l}$y$\end{tabular}}}}%
    \put(0.3734972,0.01387771){\color[rgb]{0,0,0}\makebox(0,0)[lt]{\lineheight{10.00000095}\smash{\begin{tabular}[t]{l}$x+y$\end{tabular}}}}%
    \put(0.69304854,0.01387771){\color[rgb]{0,0,0}\makebox(0,0)[lt]{\lineheight{10.00000095}\smash{\begin{tabular}[t]{l}$1$\end{tabular}}}}%
    \put(0.75850723,0.01387771){\color[rgb]{0,0,0}\makebox(0,0)[lt]{\lineheight{10.00000095}\smash{\begin{tabular}[t]{l}$x$\end{tabular}}}}%
    \put(0.81460395,0.01387771){\color[rgb]{0,0,0}\makebox(0,0)[lt]{\lineheight{10.00000095}\smash{\begin{tabular}[t]{l}$y$\end{tabular}}}}%
    \put(0.94977866,0.01387771){\color[rgb]{0,0,0}\makebox(0,0)[lt]{\lineheight{10.00000095}\smash{\begin{tabular}[t]{l}\prd\end{tabular}}}}%
    \put(0.60162863,0.01387771){\color[rgb]{0,0,0}\makebox(0,0)[lt]{\lineheight{10}\smash{\begin{tabular}[t]{l}$0$\end{tabular}}}}%
  \end{picture}%
\endgroup%
\hss}
\end{proof}

\def\RZ{(\R,+,\cdot,\Z)}
\noindent Let $\RZ$ denote the first order structure
whose domain is the set of real numbers, with the field operations and a predicate for
the subset of the integers.
It is well known that the definable
sets in~$\RZ$ coincide with the projective sets (see, for instance,
\cite[exercise~37.6]{setheory}). Therefore we get immediately the following fact.

\begin{fact} \label{alldef}
All Lie groups are definable in~$\RZ$.
\end{fact}
\begin{proof}
Given a Lie group~$G$ of dimension~$n$, by Whitney's embedding theorem,
there is an embedding of~$G$ in~$\R^{2n+1}$ as a closed
subset~$C$. The group operation of~$G$ induces an
operation~$\_\cdot\_:C\times C\to C$, which is continuous, hence a closed subset
of~$\R^{3(2n+1)}$. Therefore $C$ and~$\_\cdot\_$ are projective
sets, hence definable in~$\RZ$.
\end{proof}

\noindent We show now that one needs the full power of $\RZ$ to be
able to define all  connected Lie groups. Namely, there is a connected Lie group $G$ such that
$\RZ$ itself is definable in the group structure~$(G,\cdot)$.
 
\begin{theorem}\label{th-main}
There is a connected nilpotent Lie group that
interprets~$\RZ$. Moreover, there is a connected Lie group that
defines~$\RZ$.
\end{theorem}
\begin{proof}
First, we will construct the connected nilpotent Lie group~$G$ in which $\RZ$ is interpretable. After
that we will show how this group can be modified to obtain a group
$G'$, so that the domain~$\R$ is a definable set in~$G'$. 

Consider the following nilpotent group
\[
G = \quotient{\mathrm{H}_3\left(\R\right)}{\Gamma}
\]
Where $\mathrm{H}_3\left(\R\right)$ is the Heisenberg group
\[
\mathrm{H}_3\left(\R\right) =
\left\{\left(
\begin{array}{ccc}
1 & a & c \\
0 & 1 & b \\
0 & 0 & 1
\end{array}
\right) \st a,b,c\in\R \right\} < \mathrm{GL}_3\left(\R\right)
\]
and
\[
\Gamma =
\left\{\left(
\begin{array}{ccc}
1 & 0 & z \\
0 & 1 & 0 \\
0 & 0 & 1
\end{array}
\right) \st z\in\Z \right\} \triangleleft \mathrm{H}_3\left(\R\right)
\]
is a discrete subgroup of the center of~$\mathrm{H}_3$
\[
Z(\mathrm{H}_3) =
\left\{\left(
\begin{array}{ccc}
1 & 0 & c \\
0 & 1 & 0 \\
0 & 0 & 1
\end{array}
\right) \st c\in\R \right\}  
\]
(we choose a
particular~$\Gamma$; however, up to isomorphism, $G$ does not depend on
this choice).
For ease of notation, write $[a,b,c]$ to denote the class of the element
\[\left(
\begin{array}{ccc}
1 & a & c \\
0 & 1 & b \\
0 & 0 & 1
\end{array}
\right)
\]
of~$\mathrm{H}_3\left(\R\right)$ in the quotient.
One can check directly that $[a,b,c]=[a',b',c']$ if
and only if $a=a'$, $b=b'$, and~$c-c'\in\Z$. It follows that the centralizer
of $[a,b,c]$ in~$G$ is
\[
\mathrm{C}\left([a,b,c]\right) = \left\{[a',b',c'] \st ab'-ba'\in\Z\right\}.
\]
For any $a,b\in\R$ we can define the subgroup
\[
L_{a,b} =
\mathrm{C}\left([a,b,0]\right)
\cap
\mathrm{C}\left([\sqrt{2} a,\sqrt{2} b,0]\right)
\]
where $\sqrt{2}$ can be replaced by any irrational number, as this is enough to ensure
that
\[
L_{a,b}= \left\{[a',b',c'] \st ab'-ba' = 0 \right\}.
\]
In particular, the following are definable subgroups
of~$G$:
\begin{align*}
A &\eqdef L_{0,1} = \left\{[0,b,c] \st b,c\in\R\right\}\\
B &\eqdef L_{0,1} \cap \mathrm{C}\left([1,0,0]\right) = \left\{[0,b,c] \st b\in\Z,\;c\in\R\right\}
\end{align*}
Hence the following groups are interpretable in~$G$:
\begin{align*}
E &\eqdef \quotient{G}{Z(G)}
& R &\eqdef \quotient{A}{Z(G)}
& Z &\eqdef \quotient{B}{Z(G)}
\end{align*}
Clearly $E>R>Z$. Observe that two elements $[a,b,c]$ and~$[a',b',c']$
of~$G$ are equivalent modulo~$Z(G)$ if and only if $a=a'$ and~$b=b'$. It
follows that $E$, $R$, and~$Z$ are isomorphic respectively to~$\R^2$,
$\{0\}\times\R$, and~$\{0\}\times\Z$ through the map
$\iota\colon[a,b,c]\mapsto(a,b)\in\R^2$. These are the ingredients of our
interpretation.

The group~$R$ is the domain of the interpretation of~$\RZ$,
and $Z$ is the interpretation of~$\Z$.
To define the field operations, we will make use of Fact~\ref{coll}.
Let $\mathcal{L}$ denote the
set~$\left\{L_{a,b} \st (a,b)\in\R^2\setminus\{(0,0)\}\right\}$
of subgroups of~$G$. Clearly the
family~$\{\iota(L)\}_{L\in\mathcal{L}}$ spans all the
straight lines through~$(0,0)$, so, if the set~$\mathcal{L}$ happens to be
the image of a uniform family of definable subgroups, then we can define the collinearity relation over~$E\sim\R^2$
as follows
\[ \text{coll}(p,q,r) \;\; \stackrel{\textsl{\tiny def}}{\equiv} \;\; \exists
L\in\mathcal{L} \;\; p^{-1}q\in L\;
\wedge\; p^{-1}r \in L \]
then we conclude by Fact~\ref{coll}.
To write $\mathcal{L}$ as a uniform family, recall that the subgroups~$L_{a,b}$ are intersections of pairs of
centralizers, therefore, it suffices to find a definable predicate that
tells whether, given $g_1,g_2\in G\setminus Z(G)$, there are $a,b$ such
that $C(g_1)\cap C(g_2)=L_{a,b}$. Indeed, this happens if and only if
the group~$C(g_1)\cap C(g_2)$ is divisible by~$2$, i.e.\ for all~$x\in C(g_1)\cap C(g_2)$
there is~$y\in C(g_1)\cap C(g_2)$ such that~$x=y\cdot y$.

Now, to get a group in which $\RZ$ is definable, as opposed to
interpretable, we replace
$\mathrm{H}_3\left(\R\right)$ with the
group
\[
H' =
\left\{\left(
\begin{array}{ccc}
1 & a & c \\
0 & x & b \\
0 & 0 & 1
\end{array}
\right) \st a,b,c\in\R \wedge x\in\R^+\right\} < \mathrm{GL}_3\left(\R\right)
\]
Again $\Gamma$ is central in~$H'$, and we can construct
$G'$ as
\[
G' = \quotient{H'}{\Gamma}
\]
Now, using the notation~$[a,b,c,x]$ for elements of~$G'$, it is easy to
recover the group~$G$ as the product of the centralizers of~$[0,1,0,1]$
and~$[1,0,0,1]$. In fact
\begin{align*}
C\left([0,1,0,1]\right) &= \left\{[0,b,c,1] \st b,c\in\R \right\}\\
C\left([1,0,0,1]\right) &= \left\{[a,0,c,1] \st a,c\in\R \right\}\\
[0,b,c_1,1]\cdot[a,0,c_2,1] &= [a,b,c_1+c_2,1]
\end{align*}
Therefore we can carry out the construction of $E$, $R$,
and~$Z$ as before and get an interpretation of~$\RZ$. To turn this into an
actual definition, we only need a choice of representatives for the
elements of~$R=\quotient{A}{Z(G)}$. To this aim, let $O$
be the orbit of~$[0,1,0,1]$ under conjugation by elements
of~$C([0,0,0,2])$.
An easy computation shows that
\begin{align*}
C\left([0,0,0,2]\right) &= \left\{[0,0,c,x] \st c\in\R \wedge x\in\R^+\right\}\\
O &= \left\{[0,b,0,1] \st b\in\R^+\right\}
\end{align*}
Hence $O\cup O^{-1}\cup \left\{[0,0,0,1]\right\}$ intersects each
equivalence class of~$A$ modulo~$Z(G)$ in a single point.
\end{proof}

%\begin{corollary}
%Every Lie group is definable in the group~$G'$ of the proof above.
%\end{corollary}

%\vfill
%\pagebreak

%===================
\section{Additional remarks}

\noindent The group $G$ of Theorem~\ref{th-main} has
minimal dimension, since all connected Lie groups of dimension up to~$2$ are definable in the real field:
up to Lie isomorphism, connected 1-dimensional groups are $\SO_2(\R)$ and
$\R$, 2-dimensional groups are $\R\times\R$, $\R^{> 0} \ltimes \R$, $\SO_2(\R)
\times \R$ and $\SO_2(\R)\times\SO_2(\R)$ (see for instance
\cite[p.36]{liebook}). We don't know whether $G'$ could be replaced by a
group of dimension~$3$.

The group $G$, defined as the quotient of a connected real algebraic group by a discrete subgroup, could also be obtained as the quotient of a definably connected semialgebraic group by a definably connected $\bigvee$-definable subgroup (see \cite[Example 5.9]{survey}).

The group $G$ is
not linear (for instance, by a Theorem of Got\^o \cite[Theorem 5]{goto} the derived
subgroup of a connected solvable linear Lie group needs to intersect
trivially any maximal compact subgroup, but the derived subgroup
of~$G$
\textit{coincides} with a maximal compact subgroup).
In a private communication, Ya'acov Peterzil
observed that, by a modification of~\cite[example on p.~5]{pps3},
there is a linear group interpreting~$(\mathbb{C},+,\cdot,\Z)$---in fact, it
suffices  to repeat the same construction of the example with $\R$ replaced
by~$\mathbb{C}$. 
This raises the question of whether there is a linear
group interpreting~$(\R,+,\cdot,\Z)$.

\section*{Acknowledgements}
\noindent We thank Chris Miller for the reference
\cite{setheory}
and Ya'acov Peterzil for fruitful discussions and
suggestions.

\end{document}